\documentclass[12.5pt]{amsart}

\usepackage{amscd,amsmath,amsthm,amssymb}
\usepackage[left]{lineno}
\usepackage{color}
\usepackage{stmaryrd}

\usepackage[utf8]{inputenc}

\usepackage{cleveref}
\usepackage{epstopdf}
\usepackage{graphicx}
\usepackage{xcolor}

\usepackage{tikz}

\definecolor{verylight}{gray}{0.97}
\definecolor{light}{gray}{0.9}
\definecolor{medium}{gray}{0.85}
\definecolor{dark}{gray}{0.6}

\def\NZQ{\mathbb}               

\def\KK{{\NZQ K}}

\def\frk{\mathfrak}               

\def\mm{{\frk m}}

\def\KK{{\NZQ K}}

\renewcommand{\qedsymbol}{$\square$} 

\def\G{{\mathcal G}}

\def\eb{{\mathbf e}}
\def\ww{{\mathbf w}}
\def\0b{{\mathbf 0}}

\def\reg{{\mathbf reg}}
\def\height{\operatorname{ht}}
\def\depth{\operatorname{depth}}
\def\opn#1#2{\def#1{\operatorname{#2}}} 
\opn\chara{char} \opn\length{\ell} \opn\pd{pd} \opn\rk{rk}
\opn\projdim{proj\,dim} \opn\injdim{inj\,dim} \opn\rank{rank}
\opn\depth{depth} \opn\grade{grade} \opn\height{height}
\opn\embdim{emb\,dim} \opn\codim{codim}

\opn\Tr{Tr} \opn\bigrank{big\,rank}
\opn\superheight{superheight}\opn\lcm{lcm}
\opn\trdeg{tr\,deg}
	\opn\reg{reg} \opn\lreg{lreg} \opn\ini{in} \opn\lpd{lpd}
	\opn\size{size} \opn\sdepth{sdepth}
	\opn\link{link}\opn\fdepth{fdepth}\opn\lex{lex}
	\opn\tr{tr}
	\opn\type{type}
	\opn\gap{gap}
	\opn\arithdeg{arith-deg}
	\opn\HS{HS}
	\opn\GL{GL}
	\opn\div{div} \opn\Div{Div} \opn\cl{cl} \opn\Cl{Cl}
	\opn\Spec{Spec} \opn\Supp{Supp} \opn\supp{supp} \opn\Sing{Sing}
	\opn\Ass{Ass} \opn\Min{Min}\opn\Mon{Mon}
	\opn\Ann{Ann} \opn\Rad{Rad} \opn\Soc{Soc}\opn\Deg{Deg}
	\opn\Im{Im} \opn\Ker{Ker} \opn\Coker{Coker} \opn\Am{Am}
	\opn\Hom{Hom} \opn\Tor{Tor} \opn\Ext{Ext} \opn\End{End}
	\opn\Aut{Aut} \opn\id{id}
	
	\opn\nat{nat}
	\opn\pff{pf}
	\opn\Pf{Pf} \opn\GL{GL} \opn\SL{SL} \opn\mod{mod} \opn\ord{ord}
	\opn\Gin{Gin} \opn\Hilb{Hilb}\opn\sort{sort}
	\opn\PF{PF}\opn\Ap{Ap}
	\opn\mult{mult}
	\opn\bight{bight}
	\opn\aff{aff}
	\opn\relint{relint} \opn\st{st}
	\opn\lk{lk} \opn\cn{cn} \opn\core{core} \opn\vol{vol}  \opn\inp{inp} \opn\nilpot{nilpot}
	\opn\link{link} \opn\star{star}\opn\lex{lex}\opn\set{set}
    \opn\last{last}
	\opn\width{wd}
	\opn\Fr{F}
	\opn\QF{QF}
	\opn\G{G}
	\opn\type{type}\opn\res{res}
	\opn\conv{conv}
	\opn\Ind{Ind}
	\opn\gr{gr}
	
	\def\pot#1#2{#1[\kern-0.28ex[#2]\kern-0.28ex]}

	\opn\dirlim{\underrightarrow{\lim}}
	\opn\inivlim{\underleftarrow{\lim}}
	\let\to=\rightarrow
	
	\def\Implies{\ifmmode\Longrightarrow \else
		\unskip${}\Longrightarrow{}$\ignorespaces\fi}
	\def\implies{\ifmmode\Rightarrow \else
		\unskip${}\Rightarrow{}$\ignorespaces\fi}
	\def\iff{\ifmmode\Longleftrightarrow \else
		\unskip${}\Longleftrightarrow{}$\ignorespaces\fi}

	\let\:=\colon
	\newtheorem{Theorem}{Theorem}[section]
	\newtheorem{Lemma}[Theorem]{Lemma}
	\newtheorem{Corollary}[Theorem]{Corollary}
	
	\newtheorem{Remark}[Theorem]{Remark}

	\newtheorem{Definition}[Theorem]{Definition}

	\let\epsilon\varepsilon
	\let\kappa=\varkappa
	\def\qed{\ifhmode\textqed\fi
		\ifmmode\ifinner\quad\qedsymbol\else\dispqed\fi\fi}
	\def\textqed{\unskip\nobreak\penalty50
		\hskip2em\hbox{}\nobreak\hfil\qedsymbol
		\parfillskip=0pt \finalhyphendemerits=0}
	\def\dispqed{\rlap{\qquad\qedsymbol}}
	
	\opn\dis{dis}
	\def\pnt{{\raise0.5mm\hbox{\large\bf.}}}
	
	\opn\Lex{Lex}
	
\begin{document}
		\title {Depth of powers of integrally closed edge ideals of  edge-weighted paths}
		
	\author {Jiaxin Li }
\address{ School of Mathematical Sciences, Soochow University, Suzhou 215006, P. R. China}
\email{lijiaxinworking@163.com}

\author{Thanh Vu}
\address{Institute of Mathematics, VAST, 18 Hoang Quoc Viet, Hanoi, Vietnam}
\email{vuqthanh@gmail.com}
		
	\author{Guangjun Zhu*}
\address{ School of Mathematical Sciences, Soochow University, Suzhou 215006, P. R. China}
\email{zhuguangjun@suda.edu.cn}
\thanks{ * Corresponding author.}
	


		
		

		\begin{abstract} We compute the depth of powers of edge ideals of integrally closed edge-weighted paths.
		\end{abstract}
		\setcounter{tocdepth}{1}		
		\maketitle
		\section{Introduction}
Let $G$ be a finite, simple graph with the vertex set $V(G)$ and the edge set $E(G)$. The \emph{edge ideal} of $G$ is the quadratic monomial ideal in $S = \KK[V(G)]$ generated by all monomials $xy$ for each edge $\{x, y\} \in E(G)$. According to a celebrated theorem by Trung  \cite{T},  if $G$ is a connected simple graph, then
\[
\lim_{t\to \infty} \depth(S/I(G)^t) =
\begin{cases}
    1 & \text{if } G \text{ is bipartite},\\
    0 & \text{otherwise}.
\end{cases}
\]
In this work, we show that this phenomenon does not generally hold for edge ideals of edge-weighted paths. In general, it is known from a result of Brodmann \cite{B} that $\depth(S/I^t)$ is eventually constant for any homogeneous ideal $I$ in a polynomial ring $S$. Nonetheless, computing the depth of the powers of a given ideal $I$ remains a challenging problem, and the behavior of the depth function of powers of homogeneous ideals can be wild \cite{HNTT}. Even when $I$ is the edge ideal of a simple graph, exact values are known only for a few specific classes, such as cycles, starlike trees \cite{MTV}, and Cohen-Macaulay trees \cite{HHV}.

Suppose ${\bf w}: E(G) \rightarrow \mathbb{Z}_{>0}$ is a weight function on the edges of $G$, and we  write $G_{\bf w}$ for the pair $(G, {\bf w})$. We refer to  $G_{\bf w}$
 as an \emph{edge-weighted graph} with underlying graph $G$. The \emph{edge ideal} of  $G_{\bf w}$, introduced in \cite{PS}, is defined as
\[
I(G_{\bf w}) = \left( x^{{\bf w}(xy)} y^{{\bf w}(xy)} \mid \{x, y\} \in E(G) \right).
\]

In \cite{LTZ}, Li, Trung, and Zhu characterized the Cohen-Macaulay property for all powers of edge ideals of $G_{\bf w}$ when $G$ is a tree with a perfect matching of pendant edges. In \cite{ZLCY}, Zhu, Li, Cui, and Yang derived exact formulas for the depth of powers of edge ideals for integrally closed weighted cycles. In \cite{ZDCL}, Zhu, Duan, Cui, and Li established exact formulas for the depth of powers of edge ideals of weighted star graphs and obtained lower bounds for the depth of integrally closed edge-weighted paths. In this work, we prove that these lower bounds from \cite{ZDCL} are sharp.

For simplicity, let $\ww = (w_1, \ldots, w_{n-1})$ be a list of $n-1$ positive integers. We denote by $I(P(\ww))$ the ideal
\[
I(P(\ww)) = \left( (x_1x_2)^{w_1}, \ldots, (x_{n-1}x_n)^{w_{n-1}} \right).
\]

By \cite{DZCL,VZ}, $I(P(\ww))$ is integrally closed if and only if $\ww$ has at most two entries greater than one, referred to as non-trivial weights. Furthermore, if $\ww$ has exactly two non-trivial weights, they must occur in positions $a$ and $a+2$ for some $a$. In this article, we compute the depth of powers of the ideal $I(P(\ww))$ when it is integrally closed. Our main results are as follows:

\begin{Theorem}\label{thm_1}
Let $\ww = (1,\ldots,1,\omega,1,\ldots,1)$, where $\omega > 1$ appears in the $a$-th position,  for   some $1\le a\le \lfloor\frac{n}{2}\rfloor$. Then, for any $t\ge 2$,
\[
\depth(S/I(P(\ww))^t) =
\begin{cases}
    \max \left\{ \left\lceil \frac{n - t + 1}{3} \right\rceil ,\, 1 \right\}, & \text{if } a=1, \text{ or } t = 2, a\equiv 1\!\!\!\!\pmod 3 \text{\ and\ } n\equiv 2\!\!\!\! \pmod 4, \\[0.5em]
    \max \left\{ \left\lceil \frac{n - t}{3} \right\rceil ,\, 1 \right\}, & \text{otherwise}.
\end{cases}
\]
\end{Theorem}

\begin{Theorem}\label{thm_2}
Let $\ww = (1,\ldots,1,\gamma,1,\delta,1,\ldots,1)$, where $\gamma>1$ and $ \delta > 1$ appear at positions $a$ and $a+2$, respectively, for some $1 \le a \le \lfloor \frac{n}{2} \rfloor - 1$. Then, for any $t\ge 2$,
\[
\depth(S/I(P(\ww))^t) =
\begin{cases}
    \max \left\{ \left\lceil \frac{n - t+1}{3} \right\rceil ,\, 2 \right\}, & \text{if } a = 1, \\[0.5em]
    \max \left\{ \left\lceil \frac{n - t}{3} \right\rceil ,\, 2 \right\}, & \text{otherwise}.
\end{cases}
\]
\end{Theorem}
In particular, the limit depth is $2$ when $\ww$ has two non-trivial weights.  Section~\ref{sec:prelim} provides the basic facts and notation used throughout the paper. Section~\ref{sec:path1} proves Theorem \ref{thm_1}.Section~\ref{sec:path2}  proves Theorem \ref{thm_2}.

 \section{Preliminaries}
\label{sec:prelim}
Throughout this paper, we let $S = \KK[x_1,\ldots, x_n]$ be a standard graded polynomial ring over a field $\KK$, and let $\mm = (x_1,\ldots, x_n)$ be the maximal homogeneous ideal of $S$. For a finitely generated graded $S$-module $L$, the depth of $L$ is defined as follows:
$$\depth(L) = \min\{i \mid H_{\mm}^i(L) \ne 0\}.$$
Here $H^{i}_{\mm}(L)$ denotes the $i$-th local cohomology module of $L$ with respect to $\mm$. We use the following results throughout the paper.

\begin{Lemma}{\em (\cite[Proposition~1.2.2]{HH})}
\label{colon}
	Let $I$ and $J$ be two monomial ideals. Then,  $I : J$ is a monomial ideal, and
	\[ I : J = \bigcap_{v \in G(J)} I : (v). \]
Furthermore, the set $\{u/\gcd(u,v) : u \in G(I)\}$ is a set of minimal generators of $I : (v)$.
\end{Lemma}

\begin{Lemma}{\em (\cite[Corollary~1.3]{R})}\label{lem_upperbound} Let $I$ be a monomial ideal and let $f$ be a monomial such that $f \notin I$. Then
    $$\depth (S/I) \le \depth (S/(I:f)).$$
\end{Lemma}

		\begin{Lemma}{\em (\cite[Lemma 2.2]{HT})}
			\label{sum1}
			Let $S_{1}=\KK[x_{1},\dots,x_{m}]$ and $S_{2}=\KK[x_{m+1},\dots,x_{n}]$ be two polynomial rings  over a field $\KK$. Let $S=S_1\otimes_\KK S_2$ and  $I\subset S_{1}$,
			$J\subset S_{2}$ be two non-zero homogeneous  ideals.  Then
		$\depth(S/(I+J))=\depth(S_1/I)+\depth(S_2/J)$.
		\end{Lemma}
	
\begin{Lemma}{\em (\cite[Lemma 2.8]{M})}\label{depth_path} Let $P_n$ be a path with $n$ vertices. Then
$$\depth (S/I(P_n)) = \lceil \frac{n}{3} \rceil.$$
\end{Lemma}

We now state the following simple lemma, which generalizes \cite[Lemma 4.10]{ZDCL}.

\begin{Lemma}\label{lem_colon_leaf}
Let $G_\ww$ be an edge-weighted graph. Assume that $e=\{x_{n-1}, x_n\}$ is a leaf edge of $G$ with a trivial weight, that is, $x_n$ has a unique neighbor in $G$, namely $x_{n-1}$,  and $\ww(e) = 1$. Then, for all $t \ge 2$, we have
$$I(G_\ww)^t : (x_{n-1}x_n) = I(G_\ww)^{t-1}.$$
\end{Lemma}

\begin{proof}
Clearly, the left-hand side contains the right-hand side. According to Lemma \ref{colon},  it is sufficient to prove that if $f$ is a minimal generator of $I(G_\ww)^t$, then
$$g = f / \gcd(f, x_{n-1}x_n) \in I(G_\ww)^{t-1}.$$
Indeed, we can write $f = f_1 \cdots f_t$, where each $f_i$ is a minimal generator of $I(G_\ww)$. If $x_n$ divides $f_j$ for some $j$, then $f_j = x_{n-1} x_n$, since $x_n$ is a leaf of $G$ and $x_{n-1}$ is its unique neighbor. Hence,
$$g = f / \gcd(f, x_{n-1}x_n) =(\prod_{i=1}^{j-1}f_i)(\prod_{i=j+1}^{t}f_i) \in I(G_\ww)^{t-1}.$$

Now, assume that $x_n$ does not divide $f_j$ for any $j = 1, \ldots, t$. Then $\gcd(f, x_{n-1}x_n)$ is either $1$ or $x_{n-1}$. In the first case, $g = f \in I(G_\ww)^t$. In the second case, there exists some $j$ such that $f_j$ is divisible by $x_{n-1}$;  therefore,
$$g \text{ is divisible by } (\prod_{i=1}^{j-1}f_i)(\prod_{i=j+1}^{t}f_i).$$
The conclusion follows.
\end{proof}
\begin{Corollary}\label{cor_dec}
    Let $G_\ww$ be an edge-weighted graph. Assume that $G_\ww$ has a leaf with a trivial weight. Then, for all $t \ge 1$, $\depth (S/I(G_\ww)^t) \ge \depth (S/I(G_\ww)^{t+1})$.
\end{Corollary}
\begin{proof}
    This conclusion follows from Lemmas \ref{lem_colon_leaf} and  \ref{lem_upperbound}.
\end{proof}

By the results of \cite{ZDCL}, we can assume that $n \ge 4$ and derive the following lower bounds:

\begin{Lemma}\label{lower_bound_1} Let $\ww = (1,\ldots,1,\omega,1,\ldots,1)$ be a weight sequence with the only non-trivial weight $\omega > 1$ at the $a$-th position.   
Let $I = I(P(\ww))$. Then, for all $t \ge 2$,
$$\depth (S/I^t) \ge \begin{cases}
    \max \left\{ \left\lceil \frac{n - t + 1}{3} \right\rceil ,\, 1 \right\}, & \text{if } a = 1, \text{ or } t = 2, a\equiv 1\!\!\!\!\pmod 3 \text{\ and\ } n\equiv 2\!\!\!\! \pmod 4, \\[0.5em]
    \max \left\{ \left\lceil \frac{n - t}{3} \right\rceil,\, 1 \right\}, & \text{otherwise}.
\end{cases}$$
\end{Lemma}

\begin{Lemma}\label{lower_bound_2} Let $\ww = (1,\ldots,1,\gamma,1,\delta,1,\ldots,1)$ be a weight sequence with the only two non-trivial weights $\gamma>1$ and $\delta > 1$, at the positions $a$ and $a+2$,  respectively. 
Let $I = I(P(\ww))$. Then, for all $t \ge 2$, we have
$$\depth (S/I^t) \ge \begin{cases}
    \max \left\{ \left\lceil \frac{n - t + 1}{3} \right\rceil ,\, 2 \right\}, & \text{if } a = 1, \\[0.5em]
    \max \left\{ \left\lceil \frac{n - t}{3} \right\rceil ,\, 2 \right\}, & \text{otherwise}.
\end{cases}$$
\end{Lemma}

A graph $H$ is said to be \emph{bipartite} if there exists a decomposition  of $V(H)$ into two disjoint subsets, $X$ and $Y$, such that $E(H) \subseteq X \times Y$.
If $E(H) = X \times Y$, then $H$ is called a \emph{complete bipartite graph} and is denoted by $K_{X,Y}$.
We recall the notion of \emph{bipartite completion} introduced in \cite{MTV}.

Let $G$ be a simple graph, and let $\mathbf{e} = \{e_1, \ldots, e_t\}$ be a collection of edges of $G$, and let $N[\mathbf{e}]$
denote   the \emph{closed neighborhood} of $\mathbf{e}$ in $G$, i.e., the set of vertices that are in $e_1, \ldots, e_t$ or are neighbors of some vertex in $e_1, \ldots, e_t$.
Let $G[\mathbf{e}]$ and  $G[\overline{\mathbf{e}}]$ denote the induced subgraphs of $G$ on the sets $N[\mathbf{e}]$ and  $V(G)\setminus N[\mathbf{e}]$), respectively.

\begin{Definition}
Let $G$ be a simple graph, and let $\mathbf{e} = \{e_1, \ldots, e_t\}$ a collection of edges of $G$.
Assume that $G[\mathbf{e}]$ is bipartite with bipartition $N[\mathbf{e}] = U \cup V$.
The \emph{bipartite completion} of $G$ with respect to $\mathbf{e}$, denoted by $\widetilde{G[\mathbf{e}]}$, is the graph obtained from $G$  by adding all edges $\{u,v\}$ for $u \in U$ and $v \in V$.
\end{Definition}

By the proof of \cite[Lemma~3.1]{HHV}, we have the following.

\begin{Lemma}\label{lem_depth_bipartite_completion}
Let $G$ be a simple graph such that $G[\mathbf{e}]$ is bipartite and $G[\overline{\mathbf{e}}]$ is weakly chordal,
i.e., both  $G[\overline{\mathbf{e}}]$ and its complement have no induced cycles of length greater than $4$, and $\mathbf{e}$ is connected when viewed as a subgraph of $G$. Then
\[
\depth \left ( S /I(\widetilde{G[\mathbf{e}]}) \right ) \le 1 + \depth \left (R / I(G[\overline{\mathbf{e}}])\right ),
\]
where $R$ is the polynomial ring in the variables corresponding to $G[\overline{\mathbf{e}}]$.
\end{Lemma}

\section{The case of a weight sequence with one non-trivial weight}
		\label{sec:path1}
In this section, we will prove Theorem~\ref{thm_1}. We establish the following notation: Given a non-zero monomial $f \in S$, its \emph{support}, denoted by $\supp(f)$, is the set of indices $j$ for which  $x_j$ divides $f$, and its largest index in $\supp(f)$ is  denote by $\last(f)$. For each $j = 1, \ldots, n$, let $\deg_j(f)$ denote the exponent of $x_j$ in $f$.

In the following, we compute colon ideals of the form $I^t : f$. First, we make  some general observations.

\begin{Lemma}\label{lem_colon_incl}
Let $G_\ww$ be an edge-weighted graph. Assume that $x_n$ is a leaf of $G$ with its unique neighbor $x_{n-1}$.
Let $f$ be a non-zero monomial of $S$ such that $\supp(f) \cap \{n-1,n\} = \emptyset$.
Then, for all $t \ge 1$, we have
\[
I(G_\ww)^t : f \subseteq (I(H_\ww)^t : f) \;+\; I(G_\ww),
\]
where $H$ is the induced subgraph of $G$ on $V(G) \setminus \{x_n\}$.
\end{Lemma}

\begin{proof} According to Lemma \ref{colon},
we can assume that
\[
h \;=\; \frac{\prod_{i=1}^{t}g_i }{\gcd(\prod_{i=1}^{t}g_i, f)}
\]
is  a minimal generator of $I(G_\ww)^t : f$, where $g_1, \ldots, g_t$ are minimal generators of $I(G_\ww)$.
We can assume that $\last(g_1) \le \cdots \le \last(g_t)$.

If $\last(g_t) = n$, then $g_t \;=\; (x_{n-1} x_n)^{\ww(x_{n-1} x_n)}$. In particular, $\supp(g_t) \cap \supp(f) = \emptyset$, and thus $h$ is divisible by $g_t$. Therefore,  $h \in I(G_\ww)$.

If $\last(g_t) < n$, then $g_i \in I(H_\ww)$ for all $i = 1, \ldots, t$, so $h \in I(H_\ww)^t : f$. The conclusion follows.
\end{proof}

\begin{Lemma}\label{lem_colon_product_edges}
Let $G_\ww$ be an edge-weighted graph, and let $\eb = \{e_1, \ldots, e_{t-1}\}$ be a collection of edges of $G$.
Assume that $\eb $ is connected when viewed as a subgraph of  $G$, and that  $G[\eb]$ is bipartite.
 Assume that the weight function $\ww$ on the edges in the closed neighborhood of $\eb$ is trivial.
Then
\[
I(G_\ww)^t :  \prod_{i=1}^{t-1}e_i \;=\; I\bigl(\widetilde{G[\eb]}_{\ww'}\bigr),
\]
where the weight function $\ww'$ is the extension of $\ww$ to the edges of $\widetilde{G[\eb]}$ and  assigns trivial weights to all new edges.
\end{Lemma}

\begin{proof}
According to \cite[Lemma~2.7]{HHV}, the left-hand side contains the right-hand side.
Denote
\[
\supp(\eb) \;=\; \supp(e_1) \cup \cdots \cup \supp(e_{t-1}).
\]
According to Lemma \ref{colon},  we can assume that
\[
h \;=\;\frac{\prod_{i=1}^{t}g_i }{\gcd(\prod_{i=1}^{t}g_i, f)},
\]
where $f=e_1 \cdots e_{t-1}$ and $g_1, \ldots, g_t$ are minimal generators of $I(G_\ww)$.

If $\supp(g_i) \cap \supp(\eb) = \emptyset$ for some $i$, then $h$ is divisible by $g_i$ and thus $h \in I(G_\ww)$.
Therefore, we can assume that $\supp(g_i) \cap \supp(\eb) \neq \emptyset$ for all $i = 1, \ldots, t$.
In this case, $h$ belongs to $J^t : f$, where $J$ is the restriction of $I(G_\ww)$ to the closed neighborhood of $\eb$, and the
conclusion  follows from \cite[Lemma~2.7]{HHV}.
\end{proof}

 Let $u_j = x_j x_{j+1}$ for any $j=1,\ldots, n-1$. Let $\ww = (1, \ldots, 1, \omega, 1, \ldots, 1)$ be a sequence of $n-1$ weights, with a unique non-trivial weight $\omega$, at the $a$-th position,
where $1\le a \le \lfloor \frac{n}{2}\rfloor$. Define
\[
I = I(P(\ww)) = (u_1, \ldots, u_{a-1}, u_a^{\omega}, u_{a+1}, \ldots, u_{n-1}).
\]

We now distinguish several classes of monomials and study the corresponding quotients of powers of $I$.  By convention, we can assume that $\prod\limits_{i=p}^{q}u_i=1$ if  $p>q$.

\begin{Lemma}\label{lem_w_1_1}
Assume that  $2 \le t \le n - a - 1$. Let
\[
f_{a,t} \;=\; x_{a+2}\, u_a^\omega \prod\limits_{i=a+2}^{a+t-1}u_i.
\]
Then, for any $t\ge 2$, we have
\[
I^t : f_{a,t} \;=\begin{cases}
	I+(\bigcup\limits_{i=2}^{t+2}\{x_{i}\}), &\text{if $a=1$,}\\
	I \;+\;	(x_{a-1})\,+\, (\bigcup\limits_{i=a+1}^{a+t+1}\{x_{i}\}), &\text{if $a\ge 2$.}\\
\end{cases}
\]
In particular,
\[
\depth\!\left(S / (I^t : f_{a,t})\right)
\;=\;
1 \;+\;\left\lceil \frac{a - 2}{3} \right\rceil\;+\;
\left\lceil \frac{n - t - a - 1}{3} \right\rceil.
\]
\end{Lemma}

\begin{proof}
Let $U$ and $V$ be a bipartition of the set $\{x_{a+1},x_{a+2}, \ldots, x_{a+t+1}\}$ such that $U$ consists of vertices with odd indices and $V$ consists of vertices with even indices.
Set $L = I + I(K_{U,V})$.
By Lemma~\ref{lem_colon_product_edges},
\begin{align*}
I^t : f_{a,t}
&\supseteq (I L) : (x_{a+2}\, u_a^\omega) \supseteq I \;+\; (L : x_{a+2}) \;+\; (L : x_{a+1}^{\omega-1}) \;+\; (I : x_a^\omega) \\
&\supseteq
\begin{cases}
	I+(\bigcup\limits_{i=2}^{t+2}\{x_{i}\}),  &\text{if $a=1$,}\\
	I \;+\;	(x_{a-1})\,+\, (\bigcup\limits_{i=a+1}^{a+t+1}\{x_{i}\}),  &\text{if $a\ge 2$.}\\
\end{cases}
\end{align*}

It remains to prove the reverse inclusion
\[
I^t : f_{a,t}
\;\subseteq\;
K \;:=
\begin{cases}
	I+(\bigcup\limits_{i=2}^{t+2}\{x_{i}\}),  &\text{if $a=1$,}\\
I \;+\;	(x_{a-1})\,+\, (\bigcup\limits_{i=a+1}^{a+t+1}\{x_{i}\}),  &\text{if $a\ge 2$.}\\
\end{cases}
\]

According to Lemma \ref{colon},  we can assume that
\[
h \;=\; \frac{\prod_{i=1}^{t}g_i}{\gcd(\prod_{i=1}^{t}g_i, f_{a,t})} \;\in\; I^t : f_{a,t},
\]
where $g_1, \ldots, g_t$ are minimal generators of $I$ such that $\last(g_1) \le \last(g_2) \le \cdots \le \last(g_t)$.
There are two cases to consider.

(1) If $a=1$, then based on the assumption, $2\le t\le n-2$, which implies that $n\ge 4$. We will prove this inclusion relation by induction on $n$.

 If $n=4$, then $t=2$ and $f_{1,2}=u_1^{\omega}x_3$. If $\text{last}(g_2) = 4$, then $h$ is divisible by $x_4$,  therefore $h \in K$. If $\text{last}(g_2) < 4$, then  we set $J=(u_1^{\omega},u_2)$. Thus, $h \in (J^2 :f_{1,2})=(J:x_1^{\omega} x_2^{\omega-1})=(x_2,x_3)\subseteq K$.

In the following, we assume that  $n\ge 5$. If $n=t+2$, then $t\ge 3$, and $f_{1,t}=x_{3}\, u_1^\omega \, \prod\limits_{i=3}^{t}u_{i}$. If $\text{last}(g_t) = n$, then $h$ is divisible by $x_n$,  therefore $h \in K$. If $\text{last}(g_t) < n$, then
we set  $J=(u_1^{\omega},u_2,\ldots,u_{n-2})$. Thus
 $h \in (J^t :f_{1,t})=(J^{t-1}:f_{1,t-1})$. The conclusion follows from  the induction  hypothesis.

 If $n>t+2$, then $f_{1,t}=x_{3}\, u_1^\omega \, \prod\limits_{i=3}^{t}u_{i}$.     Using Lemma \ref{lem_colon_incl} and  the induction  hypothesis, we can conclude that $(I^t : f_{a,t}) \subseteq K$.

 (2) If $a\ge 2$, we will prove this inclusion relation by induction on $a$.

 (i) If $a=2$, then  based on the assumption, $2\le t\le n-3$, which implies that  $n\ge 5$.  We will prove this inclusion relation by induction on $n$.

 If $n=5$, then $t=2$
 and $f_{2,2}=u_2^{\omega}x_4$. If $\text{last}(g_2) = 5$, then $h$ is divisible by $x_5$, hence $h \in K$. If $\text{last}(g_2) < 5$, then  we set $J=(u_1,u_2^{\omega},u_3)$. Thus  $h \in (J^2 :f_{2,2})=(J:x_2^{\omega} x_3^{\omega-1})=(x_1,x_3,x_4)\subseteq K$.

In the following, we assume that $n\ge 6$. If $n=t+3$, then $t\ge 3$, and $f_{2,t}=x_{4}\, u_2^\omega \, \prod\limits_{i=4}^{t+1}u_{i}$. If $\text{last}(g_t) = n$, then $h$ is divisible by $x_n$, hence $h \in K$. If $\text{last}(g_t) < n$, then we set $J=(u_1,u_2^{\omega},\ldots,u_{n-2})$. Thus,   $h \in (J^t :f_{2,t})=(J^{t-1}:f_{2,t-1})$. The conclusion follows from the
induction hypothesis.

  If $n>t+3$, then $f_{2,t}=x_{4}\, u_2^\omega \, \prod\limits_{i=4}^{t+1}u_{i}$.    Using Lemma \ref{lem_colon_incl} and  the induction  hypothesis, we can derive that $(I^t : f_{a,t}) \subseteq K$.

  (ii) If $a\ge 3$, then by the fact that  $\supp( f_{a,t}) \cap \{1,2\} = \emptyset$, Lemma \ref{lem_colon_incl} and   induction  on $a$,   we can conclude that $(I^t : f_{a,t}) \subseteq K$.

The second statement follows immediately from the first, Lemma~\ref{sum1}, and Lemma~\ref{depth_path}.
\end{proof}

Similarly, we have

\begin{Lemma}\label{lem_w_1_2} Assume that $2\le t \le a-1$. Let  $g_{a,t} = x_{a-1} u_a^\omega (\prod\limits_{i=a-2}^{a-t+1}u_i)$. Then
$$I^t : g_{a,t} = I+ (x_{a+2})+(\bigcup\limits_{i=a-t}^{a}\{x_{i}\}).$$
 In particular,
$$\depth (S/I^t:g_{a,t}) = 1 + \lceil \frac{a-t-1}{3} \rceil + \lceil \frac{n-a-2}{3} \rceil.$$
\end{Lemma}
\begin{proof}
    The proof is similar to that of Lemma \ref{lem_w_1_1}.
\end{proof}

When $3\le a \le \lfloor\frac{n}{2}\rfloor$ and $3\le t \le n - 3$, we also consider the following family.
Let $t-3=c+d$ be a partition of $t-3$ such that $0\le c \le a - 3$ and $0\le d \le n - a - 3$.
Let
\[
h_{c,d} \;=\; u_a^\omega \, x_{a-1} \, x_{a+2} \cdot
\bigl(\prod\limits_{i=a-c-1}^{a-2}u_i\bigr) \cdot
\bigl(\prod\limits_{i=a+2}^{a+d+1}u_i\bigr).
\]

\begin{Lemma}\label{lem_w_1_3}
With the above notation, we have
\[
I^t : h_{c,d}
\;=\;
I \;+\; (\bigcup\limits_{i=0}^{\lceil \tfrac{c}{2} \rceil}\{x_{a-2i-1}\})
\;+\; (\bigcup\limits_{i=0}^{\lceil \tfrac{d}{2} \rceil}\{x_{a+2i+2}\})
\;+\; I(K_{U,V}),
\]
where
\[
U \;=\; \bigcup\limits_{i=a-c-2}^{a}\{x_{i}\}
\setminus \bigcup\limits_{i=0}^{\lceil \tfrac{c}{2} \rceil}\{x_{a-2i-1}\}\quad\text{and}\quad V\;=\;\bigcup\limits_{i=a+1}^{a+d+3}\{x_{i}\}
\setminus\bigcup\limits_{i=0}^{\lceil \tfrac{d}{2} \rceil}\{x_{a+2i+2}\}
\]
In particular,
\[
\depth(S / (I^t : h_{c,d}))
\;\le\;
1 \;+\; \left\lceil \frac{a - c - 3}{3} \right\rceil
\;+\; \left\lceil \frac{n - (a + d + 3)}{3} \right\rceil.
\]
\end{Lemma}

\begin{proof} It is obvious that  $I^t : h_{c,d}\supseteq I$ from the expression of $h_{c,d}$.  Let  $U_1$ and $ V_1$ be the bipartition of $\{x_a, x_{a-1}, \ldots, x_{a-c-2}\}$ into vertices with odd and even indices. Let $U_2$ and $V_2$ be the bipartition of
 $\{x_{a+1}, x_{a+2}, \ldots, x_{a+d+3}\}$
 into vertices with odd and even indices. Let
\[
J \;=\; I \;+\; I(K_{U_1,V_1}) \;+\; I(K_{U_2,V_2}).
\]
By  Lemma \ref{lem_colon_product_edges} and the fact that $u_a^\omega x_{a-1} x_{a+2} \;=\; (x_{a-1}x_a)(x_{a+1}x_{a+2})(x_a x_{a+1})^{\omega-1}$,
we have that
\[
I^t : h_{c,d}
\;\supseteq\;
J : x_a^{\omega-1} x_{a+1}^{\omega-1} \supseteq
(\bigcup\limits_{i=0}^{\lceil \tfrac{c}{2} \rceil}\{x_{a-2i-1}\})
\;+\; (\bigcup\limits_{i=0}^{\lceil \tfrac{d}{2} \rceil}\{x_{a+2i+2}\}).
\]
Note that if  $2j \le c+2$, then  $x_{a - 2j} \in I^{c+1} : (x_{a-1} \prod\limits_{i=a-c-1}^{a-2}u_{i})$. Also note that if $2k \le d+2$, then $x_{a+1+2k} \in I^{d+1} : (x_{a+2} \prod\limits_{i=a+2}^{a+d+1}u_{i})$. Thus,
 $x_{a - 2j}x_{a+1+2k}h_{c,d}\subseteq I^{c+1}\cdot I^{d+1}\cdot I=I^t$ for all $2j \le c+2$ and all $2k \le d+2$. This implies that $I(K_{U,V})\subseteq (I^t:h_{c,d})$.
Therefore,   the left-hand side contains the right-hand side.

Now we prove the reverse inclusion
\[
(I^t : h_{c,d})\subseteq L :\;=\;
I \;+\; (\bigcup\limits_{i=0}^{\lceil \tfrac{c}{2} \rceil}\{x_{a-2i-1}\})
\;+\; (\bigcup\limits_{i=0}^{\lceil \tfrac{d}{2} \rceil}\{x_{a+2i+2}\})
\;+\; I(K_{U,V}).
\]
According to Lemma \ref{colon},  we can assume that
\[
g \;=\; \frac{\prod_{i=1}^{t}f_i}{\gcd(\prod_{i=1}^{t}f_i, h_{c,d})} \;\in\; I^t : h_{c,d},
\]
where $f_1, \ldots, f_t$ are minimal generators of $I$ such that $\last(f_1) \le \last(f_2) \le \cdots \le \last(f_t)$\textcolor{blue}.
From the assumption that  $3\le t \le n - 3$, we know that
$n\ge 6$. We will prove $(I^t : h_{c,d})\subseteq L$ by induction on $n$.

(1) If $n=6$, then  $t=3$, $a=3$ since  $3\le a \le \lfloor\frac{n}{2}\rfloor$. In this case,  $c=d=0$ and  $h_{c,d}=h_{0,0}=x_2u_3^{\omega}x_5$.  There
are four cases to consider:.
(i) If $\last(f_1)=2$ and $\last(f_3)=6$, then $g\in(x_1x_6)\subseteq L$.
(ii) If $\last(f_1)=2$ and $\last(f_3)\le 5$, then
\[
g \;=\; \frac{u_1 f_2 f_3}{\gcd(u_1 f_2 f_3, x_2u_3^{\omega}x_5)}=x_1\frac{f_2 f_3}{\gcd(f_2 f_3,u_3^{\omega}x_5)},
\]
where $f_2$ and $f_3$  are minimal generators of  the ideal $J_1:=(u_1,u_2,u_3^{\omega},u_4)$. By Lemma \ref{lem_colon_leaf},  $g\in x_1(J_1^2:u_3^{\omega}x_5)=x_1(J_1:x_3^{\omega}x_4^{\omega-1})=x_1(x_2,x_4,x_5)\subseteq L$.

(iii) If $\last(f_1)\ge 3$ and $\last(f_3)=6$, then
\[
g \;=\; \frac{f_1 f_2 u_5}{\gcd(f_1 f_2 u_5, x_2u_3^{\omega}x_5)}=x_6 \frac{f_1 f_2}{\gcd(f_1 f_2,x_2u_3^{\omega})},
\]
where  $f_1$ and $f_2$  are minimal generators of  the ideal $J_2:=(u_2,u_3^{\omega},u_4,u_5)$. Again using Lemma 2.5, we get  $g\in x_6(J_2^2:x_2u_3^{\omega})=x_6(J_2:x_3^{\omega-1}x_4^{\omega})=x_6(x_2,x_3,x_5)\subseteq L$.

(iv) If $\last(f_1)\ge 3$ and $\last(f_3)\le 5$. then we set $J_3=(u_2,u_3^{\omega},u_4)$. Again using  Lemma \ref{lem_colon_leaf}, we deduce that  $g\in(J_3^3:x_2u_3^{\omega}x_5)=(J_3:(x_3x_4)^{\omega-1})\subseteq(x_2,x_5,x_3x_4)\subseteq L$.

(2) In the following, we can assume that $n\ge 7$. If $0\le c< a-3$ or $0\le d< n-a-3$, then $\supp(h_{c,d})\cap \{1,2\}=\emptyset$ or $\supp(h_{c,d})\cap \{n-1,n\}=\emptyset$. By Lemma \ref{lem_colon_incl} and the induction on $n$, we deduce that $(I^t : h_{c,d}) \subseteq L$.
In the following, we assume that $c=a-3$ and $d=n-a-3$. Thus  $t=n-3$, since $t-3=c+d$.
In this case, $ h_{c,d}=x_{a-1}u_a^{\omega}x_{a+2}(\prod\limits_{i=2}^{a-2}u_i)(\prod\limits_{i=a+2}^{n-2}u_i)$. We divide it into two  cases:

($\alpha$) If $\last(f_t)\le n-1$, then we set $J_4=(u_1,\ldots,u_{a-1},u_a^{\omega},u_{a+1},\ldots,u_{n-2})$.
By Lemma \ref{lem_colon_leaf} and the induction on $n$, we have $g\in (J_4^t:h_{a-3,n-a-3})=(J_4^{t-1}:h_{a-3,n-a-4})\subseteq L$.

($\beta$) If $\last(f_t)=n$, then we can write $g$ as
\[
g \;=\;\frac{u_{n-1}\prod_{i=1}^{t-1}f_i}{\gcd(u_{n-1}\prod_{i=1}^{t-1}f_i, h_{c,d})}\;=\; x_n \frac{\prod_{i=1}^{t-1}f_i}{\gcd(f_1 \cdots f_{t-1}, x_{n-2}h_{c, d-1})}.
\]

If $d$ is odd, then $a+2\lceil \tfrac{d}{2} \rceil+2=a+(n-a-2)+2=n$. Thus $g\in(x_{n})\subseteq (\bigcup\limits_{i=0}^{\lceil \tfrac{d}{2} \rceil}\{x_{a+2i+2}\})\subseteq L$.
If $d$ is even and  $\last(f_{t-1}) = n$, then $a+2\lceil \tfrac{d}{2} \rceil+2=a+(n-a-3)+2=n-1$. Thus $g\in(x_{n-1})\subseteq (\bigcup\limits_{i=0}^{\lceil \tfrac{d}{2} \rceil}\{x_{a+2i+2}\})\subseteq L$.
If $d$ is even and  $\last(f_{t-1})\le  n-1$, then  $g\in x_{n}(J_4^{t-1}:x_{n-2}h_{c,d-1})\subseteq I(K_{U,V})$ by induction on $n$.

We complete the proof of the first statement and the  second statement follows from the first part, Lemma~\ref{sum1}, Lemma~\ref{depth_path} and  Lemma~\ref{lem_depth_bipartite_completion}.
\end{proof}

\begin{proof}[Proof of Theorem~\ref{thm_1}]
It is sufficient to prove the claimed upper bound for the depth by Lemma~\ref{lower_bound_1}. In the following, we assume that  $I = I(P(\ww))$, $f_{a,t} \;=\; x_{a+2}\, u_a^\omega \, (\prod\limits_{i=a+2}^{t+a-1}u_i)$,  $g_{a,t} = x_{a-1} u_a^\omega (\prod\limits_{i=a-2}^{a-t+1}u_i)$ and $h_{c,d} \;=\; u_a^\omega \, x_{a-1} \, x_{a+2}
\bigl(\prod\limits_{i=a-c-1}^{a-2}u_i\bigr)
\bigl(\prod\limits_{i=a+2}^{a+d+1}u_i\bigr)$. There
are two cases to consider:

\medskip

\noindent\textbf{Case 1. } If  $1\le a \le 2$, then  for $t \le n-a-1$, we can derive by Lemma~\ref{lem_w_1_1} that
\[
\depth(S/I^t) \le \depth(S/I^t : f_{a,t}) =
\begin{cases}
1 + \lceil \dfrac{n - t - 2}{3} \rceil
= \lceil \dfrac{n - t + 1}{3} \rceil, & \text{if } a = 1, \\
1 + \lceil \dfrac{n - t - 3}{3} \rceil
= \lceil \dfrac{n - t}{3} \rceil, & \text{if } a = 2.
\end{cases}
\]
When $t > n-a-1$,  we deduce that $\depth (S/I^t) \le \depth (S/I^{n-a-1}) = 1$ by Corollary \ref{cor_dec}.

\medskip

\noindent\textbf{Case 2.}  If  $a \ge 3$. When  $t = 2$, we can deduce that
\begin{align*}
\depth(S/I^t)
& \le \min\{\depth(S/I^t : f_{a,t}),\, \depth(S/I^t : g_{a,t})\} \\
&= \min \{ \lceil \frac{a-2}{3} \rceil + \lceil \frac{n-a}{3} \rceil, \lceil \frac{a}{3} \rceil + \lceil \frac{n-a-2}{3} \rceil \}.
\end{align*}
which is equal to $\lceil \frac{n-2}{3} \rceil$ unless $a = 1 \pmod 3$ and $n = 2 \pmod 3$, which in this case is equal to $\lceil \frac{n-1}{3} \rceil$.

In the following, we can write $a$ as $a=3\ell+r$ with $r\in\{0,1,2\}$.

	If $t=3$ and $r=2$, then by Lemma \ref{lem_w_1_1}, we have
	\[
	\depth (S/I^t) \le \depth(S/I^t : f_{a,t}) =\lceil \frac{n -3}{3} \rceil.
	\]

	If $t=3$ and $r=1$, or $t=4$ and $r=2$, then by Lemma \ref{lem_w_1_2}, we have
\[
	\depth (S/I^t) \le \depth(S/I^t : g_{a,t}) =\lceil \frac{n -t}{3} \rceil.
\]

If $t=3$ and $r=0$, or $t=4$ and $r\in\{0,1\}$, or $5 \le t \le n-3$, then we choose  $c = r$ and $d = t-3-r$. Thus
by Lemma \ref{lem_w_1_3}, we have
\[
\depth (S/I^t) \le \depth(S/I^t : h_{c,d}) \le\lceil \frac{n -t}{3} \rceil.
\]

If $t > n - 3$, then by Corollary \ref{cor_dec}, we deduce that $\depth (S/I^t) \le \depth (S/I^{n-3}) = 1$. This completes the proof.
\end{proof}

   \section{The case of weight sequences with two non-trivial weights}\label{sec:path2}
In this section, we assume that
\[
\ww = (1, \ldots, 1, \gamma, 1, \delta, 1, \ldots, 1)
\]
is a weight sequence with exactly two non-trivial weights at positions $a$ and $a+2$,
where $1 \le a \le \lfloor \frac{n}{2} \rfloor- 1$.
We will consider two families of potential monomials to compute the depth of colon ideals of the form
 $I(P(\ww))^t : f$.

\begin{Lemma}\label{lem_w_2_1}
Assume that $2\le t \le n - a - 4$. Let $f_{a,t}\;=\; x_a^{\gamma - 1} x_{a+1}^\gamma (\prod\limits_{i=a+4}^{a+t+2}u_i)$.
Then
\[
I(P(\ww))^t : f_{a,t}
\;=\; I(P(\ww)) \;+\; (x_{\max\{1, a-1\}},\, x_a,\, x_{a+2}) \;+\; I(K_{U,V}),
\]
where
\[
U \;=\; \bigcup\limits_{i=1}^{\left\lfloor \frac{t+3}{2} \right\rfloor}\{x_{a+2i+1}\},
\quad\text{and}\quad
V \;=\; \bigcup\limits_{i=a+3}^{a+t+4}\{x_{i}\} \setminus U.
\]

In particular,
\[
\depth\!\left(S / (I^t : f_{a,t})\right)
\;\le \;
\left\lceil \frac{a-2}{3} \right\rceil
\;+\; 2
\;+\; \left\lceil \frac{n - (a+t+4)}{3} \right\rceil.
\]
\end{Lemma}
\begin{proof}
Note that $\prod\limits_{i=a+4}^{a+t+2}u_i \;\in\; I(P(\ww))^{t-1}$.
It follows from Lemma~\ref{lem_colon_product_edges}  that
 \begin{align*}
(I(P(\ww))^t : f_{a,t})&=(I(P(\ww)) + I(K_{U,V})) : (x_a^{\gamma-1} x_{a+1}^\gamma)\\
&=I(P(\ww))\;+\; (x_{\max\{1, a-1\}},\, x_a,\, x_{a+2}) \;+\; I(K_{U,V}).
\end{align*}
\end{proof}

\begin{Lemma}\label{lem_w_2_2} Assume that $3\le a  \le \lfloor \frac{n}{2} \rfloor- 1$ and $2\le t \le n - 6$.
Let $t-2=c+d$ be a partition of $t-2$ such that $0\le c \le a - 3$ and $0\le d \le n - a - 5$.
 Let $h_{c,d} = x_{a-1} u_a^{\gamma}\cdot \bigl(\prod\limits_{i=a-c-1}^{a-2}u_i\bigr) \cdot
	\bigl(\prod\limits_{i=a+4}^{a+d+3}u_i\bigr)$. Then
\[
	(I(P ({\bf w}))^t:h_{c,d})= I  +(\bigcup\limits_{i=a-c-2}^{a}\{x_{i}\})+(x_{a+2}) + I(K_{U,V}),
	 \]
where $U$ and $V$ are the partition of the set $\{x_{a+3},x_{a+4},\ldots,x_{a+d+5}\}$ into odd and even indices. In particular,
$$\depth (S/I^t : h_{c,d}) \le \lceil \frac{a-c-3}{3} \rceil + 2 + \lceil \frac{n - (a+d+5)}{3} \rceil.$$
    \end{Lemma}
\begin{proof} The proof is similar to that of Lemma \ref{lem_w_1_3} and we omit.
\end{proof}

\begin{Lemma}\label{t=2}
		Assume that $3\le a  \le \lfloor \frac{n}{2} \rfloor- 1$. Then
		 \[
		(I(P ({\bf w}))^2:u_2u_{a+2}))=(x_{a+1},(x_{a+2}x_{a+3})^{\delta-1},x_{a+4})+(\bigcup\limits_{p=0}^{1}\bigcup\limits_{q=1}^{2}\{x_{2p+1}x_{2q}\})+I(P ({\bf w})).
		\]
		In particular,
		\[
		\depth(S / (I(P ({\bf w}))^2 : u_2u_{a+2}))
		\;\le\;
		2 \;+\; \left\lceil \frac{a - 4}{3} \right\rceil
		\;+\; \left\lceil \frac{n - (a+4)}{3} \right\rceil.
		\]
\end{Lemma}
\begin{proof}
	Let $I = I(P(\ww))$ and $g=u_2u_{a+2}$. Then by Lemma~\ref{lem_colon_product_edges}, we deduce that
	\begin{align*}
		(I^2 : g)&=(I + (\bigcup\limits_{p=0}^{1}\bigcup\limits_{q=1}^{2}\{x_{2p+1}x_{2q}\})) : (x_a^{\gamma-1} x_{a+1}^\gamma)\\
		&=(x_{a+1},(x_{a+2}x_{a+3})^{\delta-1},x_{a+4})+(\bigcup\limits_{p=0}^{1}\bigcup\limits_{q=1}^{2}\{x_{2p+1}x_{2q}\})+I.
	\end{align*}
\end{proof}

\begin{proof}[Proof of Theorem \ref{thm_2}] It is sufficient to prove the claimed upper bound for the depth by Lemma~\ref{lower_bound_2}. In the following, we assume that   $I = I(P(\ww))$, $f_{a,t}\;=\; x_a^{\gamma - 1} x_{a+1}^\gamma (\prod\limits_{i=a+4}^{a+t+2}u_i)$, $g=u_2u_{a+2}$ and $h_{c,d} = x_{a-1} u_a^{\gamma}\cdot \bigl(\prod\limits_{i=a-c-1}^{a-2}u_i\bigr) \cdot \bigl(\prod\limits_{i=a+4}^{a+d+3}u_i\bigr)$. There
are two cases to consider:

\medskip

\noindent\textbf{Case 1.} If  $a \le 2$, then
for $2 \le t \le n-a-4$, we can derive by Lemma~\ref{lem_w_2_1} that
\[
\depth(S/I^t) \le \depth(S/(I^t : f_{a,t})) = 2 + \lceil \frac{n - (a+t+4)}{3} \rceil
= \begin{cases}
\lceil \dfrac{n - t + 1}{3} \rceil, & \text{if } a = 1, \\
\lceil \dfrac{n - t}{3} \rceil, & \text{if } a = 2.
\end{cases}
\]
When $t > n-a-4$, then by Corollary \ref{cor_dec}, we deduce that $\depth (S/I^t) \le \depth (S/I^{n-a-4}) = 2$.

\medskip

\noindent\textbf{Case 2. }  If $a \ge 3$. Then we can write $a$ as $a=3\ell+r$ with $r\in\{0,1,2\}$.

If $t=2$ and $r=1$. Then by Lemma \ref{t=2}, we have
\[
\depth (S/I^t) \le \depth(S/(I^t : g))\le \lceil \frac{n -2}{3} \rceil.
\]

If $t=2$ and $r=2$, or $t=3$ and $r=2$. Then by Lemma \ref{lem_w_2_1}, we have
		\[
		\depth (S/I^t) \le \depth(S/(I^t : f_{a,t}))\le \lceil \frac{n -t}{3} \rceil.
		\]

If $t=2$ and $r=0$, or $t=3$ and $r\in\{0,1\}$, or $4\le t\le n-6$. Then we choose  $c = r$ and $d = t-2 -c$. Thus by Lemma \ref{lem_w_2_2}, we have

\[
\depth (S/I^t) \le \depth(S/(I^t : h_{c,d})) = 2 + (\ell-1) + \lceil \frac{n - (3 \ell+t+3)}{3} \rceil
\le \lceil \frac{n -t}{3} \rceil.
\]

If $t > n - 6$, then by Corollary \ref{cor_dec}, we deduce that $\depth (S/I^t) \le \depth (S/I^{n-6}) = 2$. This completes the proof.
\end{proof}

\begin{Remark} If  $\ww$ has two non-trivial weights and $I(P(\ww))$ is not integrally closed, then the limit depth of $I(P(\ww))$ may be $1$. For example, let $\ww = (2,1,1,2)$.  It can be verified that
\[
\depth (S / I(P(\ww))^t ) = 1 \quad \text{for all } t \ge 4.
\]
Thus, in general, computing the exact value of depth of powers of edge ideals of arbitrary edge-weighted paths is a non-trivial problem.
\end{Remark}

\medskip
\hspace{-6mm} {\bf Acknowledgments}

 \vspace{3mm}
\hspace{-6mm}  This research is supported by the Natural Science Foundation of Jiangsu Province
(No. BK20221353) and the National Natural Science Foundation of China (No.
12471246). The authors are grateful to the software systems \cite{Co} and \cite{GS}
 for providing us with a large number of examples.

\medskip
\hspace{-6mm} {\bf Data availability statement}

\vspace{3mm}
\hspace{-6mm}  The data used to support the findings of this study are included within the article.

\medskip
\hspace{-6mm} {\bf Conflict of interest}

\vspace{3mm}
\hspace{-6mm}  The authors declare that they have no competing interests.

	\end{document}